\documentclass[10pt,a4paper]{article}
\usepackage[latin1]{inputenc}
\usepackage{amsmath}
\usepackage{amsfonts}
\usepackage{amssymb,amsbsy,amsthm}

\usepackage{enumerate}
\author{Mohamed Khaled and Tarek Sayed Ahmed\\ Department of Mathematics, Faculty of Science,\\
Cairo University, Giza, Egypt.}
\title{The Andreka-Resek-Thompson and Ferenczi's results using games}
\date{}
\newtheorem{thm}{Theorem}[section]

\newtheorem{cl}{Claim}[section]

\newtheorem{lem}{Lemma}[section]
\newtheorem*{remark}{Remark}
\newtheorem{defn}{Definition}[section]
\newenvironment{athm}[1]{\par\noindent
{\bf #1 }. \slshape }
{\upshape\par}

\def\A{{\mathfrak{A}}}
\def\B{{\mathfrak{B}}}
\def\C{{\mathfrak{C}}}
\def\At{{\sf At}}
\def\d{{\sf d}}
\def\c{{\sf c}}
\def\s{{\sf s}}
\def\t{{\sf t}}
\def\At{{\sf At}}
\def\C{{\sf C}}
\def\D{{\sf D}}
\def\V{{\sf V}}
\def\S{{\sf S}}
\def\sA{{\sf A}}
\def\si{{i_0, i_1, \cdots, i_k}}
\def\sj{{j_0, j_1, \cdots, j_k}}
\def\tr{{[i_0|j_0]|[i_1|j_1]|\cdots |[i_k|j_k]}}
\def\tt{{(\t^{j_0}_{i_0}\cdots \t^{j_k}_{i_k})^{(\A)}}}
\def\ttr{{(\t^{j_0}_{i_0}\cdots \t^{j_k}_{i_k})^{(\mathfrak{Rd}_{pt}\A)}}}
\begin{document}
\maketitle
\begin{abstract}
We provide a new proof of the celebrated Andr\'{e}ka-Resek-Thompson representability result of certain finitely 
axiomtaized cylindric-like algebras, together with its quasipolyadic equality analogue proved 
by Ferenzci. Our proof uses games as introduced in algebraic logic by Hirsch and Hodkinson. 
Using the same method, we prove a new representability result for diagaonal free reducts 
of such algebras. Such representability results provides completeness theorems for variants of first order logic, 
that can also be viewed as multi-modal logics.
Finally, using a result of Marx, 
we show that all varieties considered enjoy the superamalgmation property, a strong form of amalgamation, implying that such logics also enjoy a Craig 
interpolation theorem.
\end{abstract}
\section{Introduction}
Stone's representation theorem for Boolean algebras can be formulated in two, essentially equivalent ways. 
Every Boolean algebra is isomorphic to a field of sets,
or the class of Boolean set algebras can be axiomatized by a finite set of equations. As is well known, Boolean algebras 
constitute the algebraic counterpart of propositional logic.
Stone's representation theorem, on the other hand, is the algebraic equivalent of the completeness theorem for propositional logic.

However, when we step inside the realm of first order logic, things tend to become more complicated. 
Not every abstract cylindric algebra is representable as a field of sets, where
the extra Boolean operations of cylindrifiers and diagonal elements are faithfully represented by projections and equality.
Disappointingly, the class of representable algebras fail to be axiomatized  by any reasonable finite schema and its resistance to such 
axiomatizations is inevitable. This is basically a reflection of the essential incompleteness of natural (more basic) infinitary 
extensions of first order logic. In such extensions, unlike first order logic, validity cannot be captured by a finite schema. 

Such extentions are obtained by dropping the condition of local finiteness 
(reflecting the simple fact that first order formulas contain only finitely many variables)
in algebras considered, 
allowing formulas of infinite length. This is necessary if we want to deal with the so-called algebriasable extensions of first order logic; 
extensions that are akin to universal algebraic investigations.

The condition of local finiteness, cannot be expressed in first order logic, 
and this is not warranted if we want to deal, like in the case of Boolean algebras, only with equations, 
or at worst quasi-equations.  Then we are faced with the following problem. Find a simple (hopefully finitary) axiomatization 
of classes of representable algebras abounding in algebraic logic, using only equations or quasi equations, 
which also means that we want to stay in the realm of quasivarieties.

There are two conflicting but complementary facets 
of such a problem, referred to in the literature, as the representation problem.
One is to delve deeply in investigating the complexity of potential axiomatizations for existing varieties 
of representable algebras, the other is to try to sidestep 
such wild unruly complex axiomatizations, often referred to as {\it taming methods}. 
Those taming methods can either involve passing to (better behaved) expansions of the algebras considered,
or else change the very  notion of representatiblity involved, as long as it remains concrete enough.
The borderlines are difficult to draw, we do might not know what is {\it not} concrete enough, but we can
 judge that a given representability notion is satisfactory, once we have one. 
(This is analogous to  undecidability issues, with the main difference that we do know what we mean 
by {\it not decidable}. We do not have an analogue of a 'recursive representability notion'). 

One of the taming methods is {\it relativization}, meaning that we search for representations on sets consisting of arbitrary $\alpha$ sequences($\alpha$ an ordinal specifying the dimension
of algebras considered), rather than squares, that is set of the form
$^{\alpha}U$ for some set $U$. It turns out that this can be done when do
not insist on commutativity of cylindrifications. Dropping
commutativity makes life much easier in many respects, not only
representability. An example is decidability of the equational theory
of the class of algebras in question. Typically given a set of equations $\Sigma$, show that
$\A\models \Sigma$ iff $\A$ is representable as an algebra whose
elements are genuine relations and operations are set theoretic operations pending only on manipulations of concrete relations.
Very few positive results are known in this regard, the most famous is the Resek - Thompson celebrated theorem proved in \cite{an}. The Resek - Thompson result is a refinement of a result of Resek due to Thompson. The first proof
of Resek's result (that is slightly different) was more than 100 pages long.
The short proof of the modified Resek - Thompson result in \cite{an} is due to Andr\'eka.
In this paper we provide also a relatively short  proof of this theorem using games as introduced in algebraic logic by Hirsch and Hodkinson \cite{hh}.

In  \cite{f}, Ferenczi talks about an important case for fields of sets that occur when the unit consists of a certain set of $\alpha$-sequences. In addition to the usual set theoretic boolean operations, the $i$th cylindrification and the constants $ij$th diagonal, new natural operations are imposed  to describe such field of sets. Such operations are e.g. the elementary substitution $[i|j]$ and the elementary transposition $[i,j]$ for every $i,j<\alpha$, restricted to the unit. Ferenczi considers the extended field of sets which is closed under these operations and then gives a positive answer to the question: Do these fields of sets form a variety, and if so, what is its axiomatization?\\ Again, in this paper we provide a shorter proof of Ferenczi's result, which provides a \textit{finite} axiomatization, using games as introduced in \cite{hh}. We follow the axiomatization provided by Ferenczi in his recent paper \cite{f}. We refer the reader to \cite{f} and \cite{an} to get a grasp of the importance of this problem in algebraic logic.

\par Building representations can be implemented by the step-by-step method (as in \cite{an} and \cite{f}), which consists of treating defects one by one and then taking a limit where the contradictions disappear. What can be done by step-by-step constructions, can be done by games but not the other way round. Games were introduced in algebraic logic by Hirsch and Hodkinson. Such games, which are basically Banach-Mazur games in disguise, are games of infinite lengths between two players $\forall$ and $\exists$. 
The real advantage of the game technique is that games do not only build representations, when we know that such representations exist, but they also tell us  when such representations exist, if we do not know a priori that they do. The translation however from step-by step techniques to games is not always a 
purely mechanical process, even if we know that it can be done. This transfer can well involve some ingenuity, in obtaining games are transparent, intuitive and easy to grasp. It is an unsettled (philosophical) question as to which is more intuitive, step-by-step techniques or games. Basically this depends on the context, but in all cases it is nice to have both available if possible, when we know one exists. When we have a step-by-step technique, then we are sure that there is 
at least one corresponding game. Choosing a simple game is what counts at the end.

\par We follow the notation and terminology of \cite{an} and \cite{f}. In particular, for relations $R,S$ $R|S=\{(a,b): \exists c[(a,c)\in R, (c,b)\in S]\}$.
\begin{defn}
\begin{description}

\item[Class $Crs_{\alpha}$]
An algebra $\A$ is a cylindric relativised set algebra of dimension $\alpha$ with unit $\V$ if it is of the form $$\langle\sA, \cup, \cap, \sim_{\V}, \phi, \V, \C_{i}^{\V}, \D_{ij}^{\V}\rangle_{i,j<\alpha},$$ where $\V$ is a set of $\alpha$-termed sequences such that $\sA$ is a non-empty set of subsets of $\V$, closed under the Boolean operations $\cup$, $\cap$, $\sim_{\V}$ and under the cylindrifications $\C^{\V}_{i}X=\{y\in\V:y^{i}_{u}\in X\text{ for some }u\}$, where $i<\alpha$, $X\in \sA$, and $\sA$ contain the elements $\phi$, $\V$ and the diagonals $\D^{\V}_{ij}=\{y\in\V:y_i=y_j\}$.
\item[Class $Drs_{\alpha}$]
An algebra $\A$ is a diagonal free relativised set algebra of dimension $\alpha$ with unit $\V$ if it is of the form $$\langle\sA, \cup, \cap, \sim_{\V}, \phi, \V, \C_{i}^{\V}\rangle_{i,j<\alpha},$$ where $\V$ is a set of $\alpha$-termed sequences such that $\sA$ is a non-empty set of subsets of $\V$, closed under the Boolean operations $\cup$, $\cap$, $\sim_{\V}$ and under the cylindrifications $\C^{\V}_{i}X=\{y\in\V:y^{i}_{u}\in X\text{ for some }u\}$, where $i<\alpha$, $X\in \sA$, and $\sA$ contain the elements $\phi$ and $\V$.
\end{description}
\end{defn}
The meaning of the notation $y^{i}_{u}$ is $(y^{i}_{u})_{j}=y_j$ if $j\not=i$ and $(y^{i}_{u})_j=u$ if $j=i$.
\begin{athm}{Some concepts and notation concerning $Crs_{\alpha}$ and $Drs_{\alpha}$:}
\begin{enumerate}[-]
\item The class $Crs_{\alpha}$ is a subclass of $Drs_{\alpha}$.
\item Let $\A\in Crs_{\alpha}$ be with unit element $\V$. The relativized substitution operator\footnote{Clearly $\sA$ is closed under ${^{\V}\S^{i}_{j}}$, since $\sA$ is closed under the cylindrifications and contains the diagonals.} ${^{\V}\S^{i}_{j}}$ is defined as $${^{\V}\S^{i}_{j}}X=\C^{\V}_{i}(\D^{\V}_{ij}\cap X) \text{ }\text{ } (X\in\sA).$$
We often omit the superscript $\V$ from ${^{\V}\S^{i}_{j}}$ and write $\S^{i}_{j}$.
\item The transformation $\tau$ defined on $\alpha$ is called finite if $\tau i=i$ except for finitely many $i<\alpha$. Important special cases of the finite transformations are the transformations $[i|j]$, called elementary substitutions, and $[i,j]$, called transpositions.
\item Let $\A\in Drs_{\alpha}$ be with unit element $\V$. If $y\in\V$ and $\tau$ is any finite transformation on $\alpha$, $^{\V}\S_{\tau}y$ is defined as $\tau|y$. If $X\in\sA$, then $^{\V}\S_{\tau}X$ is defined as $\{\tau|y:y\in X\}$\footnote{$\sA$ need not to be closed under $^{\V}\S_{\tau}$ for any arbitrary $\tau$.}.
\end{enumerate}
\end{athm}
\begin{defn}[Class $D_{\alpha}$]It is the subclass of $Crs_{\alpha}$ for which ${^{\V}\S^{i}_{j}}V=V$ for every $i, j < \alpha$, where $V$ is the unit of the algebra.
\end{defn}
If we consider the substitution operators ${^{\V}\S_{[i|j]}}$ and the transposition operators${^{\V}\S_{[i,j]}}$ defined above, then we get the polyadic versions of $Crs_{\alpha}$.
\begin{defn}
\begin{description}
\item[Class $Prs_{\alpha}$]
An algebra $$\mathfrak{B}=\langle\sA, \cup, \cap, \sim_{\V}, \phi, \V, \C^{\V}_{i}, {^{\V}\S_{[i|j]}}, {^{\V}\S_{[i,j]}}\rangle_{i,j<\alpha}$$
is an $\alpha$-dimensional polyadic relativized set algebra if for the diagonal free reduct, $\mathfrak{Rd}_{df}\mathfrak{B}\in Drs_{\alpha}$; further, $\sA$ is closed under the transpositions ${^{\V}\S_{[i,j]}}$ and the substitutions ${^{\V}\S_{[i|j]}}$.
\item[Class $Pers_{\alpha}$] An algebra $$\mathfrak{B}=\langle\sA, \cup, \cap, \sim_{\V}, \phi, \V, \C^{\V}_{i}, {^{\V}\S_{[i|j]}}, {^{\V}\S_{[i,j]}}, \D^{\V}_{ij}\rangle_{i,j<\alpha}$$
is an $\alpha$-dimensional polyadic equality relativized set algebra if for the cylindric reduct, $\mathfrak{Rd}_{ca}\mathfrak{B}\in Crs_{\alpha}$; further, $\sA$ is closed under the transpositions ${^{\V}\S_{[i,j]}}$\footnote{Clearly, $\sA$ is closed under the substitutions ${^{\V}\S_{[i|j]}}$ because ${^{\V}\S_{[i|j]}}={^{\V}\S^{i}_{j}}$.}.
\item[Class $Srs_{\alpha}$] An algebra $$\mathfrak{B}=\langle\sA, \cup, \cap, \sim_{\V}, \phi, \V, \C^{\V}_{i}, {^{\V}\S_{[i|j]}}\rangle_{i,j<\alpha}$$
is an $\alpha$-dimensional substitution relativized set algebra if for the diagonal free reduct, $\mathfrak{Rd}_{df}\mathfrak{B}\in Drs_{\alpha}$; further, $\sA$ is closed under the substitutions ${^{\V}\S_{[i|j]}}$.
\end{description}
\end{defn}
\begin{defn}[Class $Dp_{\alpha}$]It is the subclass of $Prs_{\alpha}$ for which ${^{\V}\S_{[i|j]}}V=V$ for every $i, j < \alpha$, where $V$ is the unit of the algebra.
\end{defn}
\begin{defn}[Class $Dpe_{\alpha}$]It is the subclass of $Pers_{\alpha}$ for which ${^{\V}\S_{[i|j]}}V=V$ for every $i, j < \alpha$, where $V$ is the unit of the algebra.
\end{defn}
\begin{defn}[Class $Ds_{\alpha}$]It is the subclass of $Srs_{\alpha}$ for which ${^{\V}\S_{[i|j]}}V=V$ for every $i, j < \alpha$, where $V$ is the unit of the algebra.
\end{defn}

\par We assume the knowledge of the concepts of cylindric algebras \cite{hmt1}. The cylindric axiom \begin{description}
\item[$(C_{4})$] $\c_i\c_j x=\c_j\c_i x$
\end{description}
proved to be quite a strong property. In the following axiomatization this property is replaced by a weaker property. Furthermore, what is called merry-go-round axioms ($\textbf{MGR}$) are postulated. By Resek-Thompson theorem \cite{an}, the existence of such axioms yields representability by relativized set algebra. The following axiomatization is due to Thompson and Andreka.
\begin{defn}[Class $PTA_{\alpha}$]
An algebra $\A=\langle A, +, \cdot, -, 0, 1, \c_{i}, \d_{ij}\rangle_{i,j\in\alpha}$, where $+$, $\cdot$ are binary operations, $-$, $\c_i$ are unary operations and $0$, $1$, $\d_{ij}$ are constants for every $i,j\in\alpha$, is partial transposition algebra\footnote{We call it partial transposition algebra to illustrate that MGR axiom gives partial transpositions. This class is different from the class of partial transposition algebras defined in \cite{ptaf}.} if it satisfies the following identities for every $i, j, k\in\alpha$.
\begin{description}
\item[$(C_{0})-(C_{3})$] $\langle A, +, \cdot, -, 0, 1, \c_i\rangle_{i\in\alpha}$ is a Boolean algebra with additive closure operators $\c_i$ such that the complements of $\c_i$-closed elements are $\c_i$-closed,
 \item[$(C_4)^*$] $\c_i\c_j x\geq \c_j\c_i x\cdot \d_{jk}$ if $k\notin\{i, j\}$,
\item[$(C_5)$]$\d_{ii}=1$,
\item[$(C_6)$]$\d_{ij}=\c_k(\d_{ik}\cdot \d_{kj})$ if $k\notin\{i, j\}$,
\item[$(C_7)$]$\d_{ij}\cdot\c_i(\d_{ij}\cdot x)\leq x$ if $i\not= j$,
\item[(MGR)]for every $i,j\in\alpha$, $i\not=j$, let $s^i_j x=\c_i(\d_{ij}\cdot x)$, $s^i_i x=x$. Then:
$$ {\sf s}_i^k{\sf s}_j^i{\sf s}_m^j{\sf s}_k^m{\sf c}_kx={\sf s}_m^k{\sf s}_i^m{\sf s}_j^i {\sf s}_k^j{\sf c}_kx\text{ if }k\notin \{i,j,m\}, m\notin \{i,j\}$$.
 \end{description}
\end{defn}
We also assume the basic knowledge of the concepts of polyadic equality and quasi-polyadic equality algebras \cite[p. 266]{hmt2}. The so called finitary polyadic equality algebras, i.e., the class $FPEA_{\alpha}$, is term defnitionally equivalent to the quasi-polyadic equality algebras \cite[Theorem 1]{st}. The axiomatization of $FPEA_{\alpha}$ and of the class $TEA_{\alpha}$ to be introduced are different in only one axiom, namely the axiom
\begin{description}
\item[($F_5$)] $\s^i_j\c_k x=\c_k\s^{i}_{j}x$ if $k\notin\{i,j\}$.
\end{description}
The following axiomatization is due to Ferenczi, abstracting
away from the class $Pers_{\alpha}$ (meaning that the axioms all hold in
$Pers_{\alpha}$). This
is a soundness condition. Ferenzci proves completeness of these
axioms, which we also prove using the different technique of resorting
to games.
\begin{defn}[Class $TEA_{\alpha}$]
A transposition equality algebra of dimension $\alpha$ is an algebra$$\A=\langle A, +, \cdot, -, 0, 1, \c_i, \s^{i}_{j}, \s_{ij}, \d_{ij}\rangle_{i,j\in\alpha},$$ where $\c_i, \s^{i}_{j}, \s_{ij}$ are unary operations, $\d_{ij}$ are constants, the axioms ($F_0$)-($F_9$) below are valid for every $i,j,k<\alpha$:
\begin{description}
\item[$(Fe_0)$] $\langle A, +, \cdot, -, 0, 1\rangle$ is a boolean algebra, $\s^i_i=\s_{ii}=\d_{ii}=Id\upharpoonright A$ and $\s_{ij}=\s_{ji}$,
\item[$(Fe_1)$] $x\leq\c_i x$,
\item[$(Fe_2)$] $\c_i(x+y)=\c_i x+\c_i y$,
\item[$(Fe_3)$] $\s^{i}_{j}\c_i x=\c_i x$,
\item[$(Fe_4)$] $\c_i\s^{i}_{j}x=\s^{i}_{j}x$, $i\not=j$,
\item[$(Fe_5)^*$] $\s^{i}_{j}\s^{k}_{m}x=\s^{k}_{m}\s^{i}_{j}x$ if $i,j\notin\{k,m\}$,
\item[$(Fe_6)$] $\s^{i}_{j}$ and $\s_{ij}$ are boolean endomorphisms,
\item[$(Fe_7)$] $\s_{ij}\s_{ij}x=x$,
\item[$(Fe_8)$] $\s_{ij}\s_{ik}x=\s_{jk}\s_{ij}x$, $i,j,k$ are distinct,
\item[$(Fe_9)$] $\s_{ij}\s^{i}_{j}x=\s^{j}_{i}x$,
\item[$(Fe_{10})$] $\s^{i}_{j}\d_{ij}=1$,
\item[$(Fe_{11})$] $x\cdot\d_{ij}\leq\s^{i}_{j}x$.
\end{description}
For $\A\in TA_{\alpha}$, its partial transposition reduct is the structure $$\mathfrak{Rd}_{pt}\A=\langle A, +, \cdot, -, 0, 1, \c_i, \s^{i}_{j}, \d_{ij}\rangle_{i,j\in\alpha}.$$
\end{defn}
The following axiomatization is new, it is obtained from Ferenczi's axiomatization by dropping equations
involving diagonal elements.
\begin{defn}[Class $TA_{\alpha}$]
A transposition algebra of dimension $\alpha$ is an algebra$$\A=\langle A, +, \cdot, -, 0, 1, \c_i, \s^{i}_{j}, \s_{ij}\rangle_{i,j\in\alpha},$$ where $\c_i, \s^{i}_{j}, \s_{ij}$ are unary operations, the axioms ($F_0$)-($F_9$) below are valid for every $i,j,k<\alpha$:
\begin{description}
\item[$(F_0)$] $\langle A, +, \cdot, -, 0, 1\rangle$ is a boolean algebra, $\s^i_i=\s_{ii}=\d_{ii}=Id\upharpoonright A$ and $\s_{ij}=\s_{ji}$,
\item[$(F_1)$] $x\leq\c_i x$,
\item[$(F_2)$] $\c_i(x+y)=\c_i x+\c_i y$,
\item[$(F_3)$] $\s^{i}_{j}\c_i x=\c_i x$,
\item[$(F_4)$] $\c_i\s^{i}_{j}x=\s^{i}_{j}x$, $i\not=j$,
\item[$(F_5)^*$] $\s^{i}_{j}\s^{k}_{m}x=\s^{k}_{m}\s^{i}_{j}x$ if $i,j\notin\{k,m\}$,
\item[$(F_6)$] $\s^{i}_{j}$ and $\s_{ij}$ are boolean endomorphisms,
\item[$(F_7)$] $\s_{ij}\s_{ij}x=x$,
\item[$(F_8)$] $\s_{ij}\s_{ik}x=\s_{jk}\s_{ij}x$, $i,j,k$ are distinct,
\item[$(F_9)$] $\s_{ij}\s^{i}_{j}x=\s^{j}_{i}x$.
\end{description}
\end{defn}

The following axiomatization is new. It is similar to Pinter's axiomatization \cite{pinter} with two major differences. We do not have commutativity of cylindrifications, this is one thing; the other
is that we stipulate the $MGR$ identities.

\begin{defn}[Class $SA_{\alpha}$]
A substitution algebra of dimension $\alpha$ is an algebra$$\A=\langle A, +, \cdot, -, 0, 1, \c_i, \s^{i}_{j}\rangle_{i,j\in\alpha},$$ where $\c_i, \s^{i}_{j}$ are unary operations, the axioms ($S_0$)-($F_8$) below are valid for every $i,j,k<\alpha$:
\begin{description}
\item[$(S_0)$] $\langle A, +, \cdot, -, 0, 1\rangle$ is a boolean algebra and $\s^i_i=Id\upharpoonright A$,
\item[$(S_1)$] $x\leq\c_i x$,
\item[$(S_2)$] $\c_i(x+y)=\c_i x+\c_i y$,
\item[$(S_3)$] $\s^{i}_{j}\c_i x=\c_i x$,
\item[$(S_4)$] $\c_i\s^{i}_{j}x=\s^{i}_{j}x$, $i\not=j$,
\item[$(S_5)^*$] $\s^{i}_{j}\s^{k}_{m}x=\s^{k}_{m}\s^{i}_{j}x$ if $i,j\notin\{k,m\}$,
\item[$(S_6)$] $\s^{i}_{j}$ is boolean endomorphism,
\item[$(S_7)$] $\s^{k}_{k}\s^j_k x=\s^k_i\s^j_i x$,
\item[$(S_8)$] $$ {\sf s}_i^k{\sf s}_j^i{\sf s}_m^j{\sf s}_k^m{\sf c}_kx={\sf s}_m^k{\sf s}_i^m{\sf s}_j^i {\sf s}_k^j{\sf c}_kx\text{ if }k\notin \{i,j,m\}, m\notin \{i,j\}$$.
\end{description}
\end{defn}

\par $(F_5)^*$ (and also $(F_5)^*$ and $(S_5)^*$) is obviously a weakening of $(F_5)$. Also it is known that $\mathfrak{Rd}_{pt}\A\in PTA_{\alpha}$, for any $\A\in TEA_{\alpha}$ \cite{st}. We consider as known the concept of the substitution operator $\s_{\tau}$ defined for any finite transformation $\tau$ on $\alpha$; $\s_{\tau}$ can be introduced uniquely in $FPEA_{\alpha}$ and in $TEA_{\alpha}$, too. The existence of such an $\s_{\tau}$ follows from the proof of \cite[Theorem 1(ii)]{st}, it is easy to check that the proof works by assuming $(F_5)^*$ instead of $(F_5)$ and (notationally) the composition operator $|$ instead of $\circ$.
\par Throughout this paper we assume that the polyadic-like algebras occurring here are equipped with the operator $\s_{\tau}$, where $\tau$ is finite. Further, $\s_{\tau}$ is assumed to have the following properties for arbitrary finite transformations $\tau$ and $\lambda$ and ordinals $i,j<\alpha$ (by \cite[p.542]{st}):
\begin{enumerate}[]
\item $\s_{\tau|\lambda}=\s_{\tau}\s_{\lambda}$\footnote{This depends on ($F_7$) and ($F_8$).},
\item$\s^{i}_{j}=\s_{[i|j]}$,
\item $\s_{\tau}\d_{ij}=\d_{\tau i\tau j}$ (of course only in the class $TEA_{\alpha}$),
\item $\c_i\s_{\tau}\leq\s_{\tau}\c_{\tau-1}$, here $\tau$ is finite permutation.
\end{enumerate}
\begin{lem}\label{lesa}
Let $\alpha$ be an ordinal, $\A\in TEA_{\alpha}$, $a\in\At\A$ and $i,j\in\alpha$. Then $$a\leq\d_{ij}\Longrightarrow\s_{[i,j]}a=a.$$
\end{lem}
\begin{proof}
See \cite[p. 875]{f}.
\end{proof}
\section{Games and Networks}
In this section fix $n\in\omega$.
We start with some preparations. Let $\overline{x}$, $\overline{y}$ be $n$-tuples of elements of some set. We write $x_{i}$ for the ith element of $\overline{x}$, for $i<n$, so that $\overline{x}=(x_{0}, \cdots, x_{n-1})$. For $i<n$, we write $\overline{x}\equiv_{i}\overline{y}$
if $x_{j}=y_{j}$ for all $j<n$ with $j\not=i$. The next two definitions are taken from \cite{hh}. A (relativized) network is a finite
approximation to a (relativized) representation.
\begin{defn} \
\begin{itemize}
\item Let $\A\in PTA_{n}$. A relativized $\A$ pre-network is a pair $N=(N_1, N_2)$
where $N_1$ is a finite set of nodes $N_2:N_1^n\to \A$ is a partial map, such that if $f\in domN_2$,
and $i,j<n$ then $f_{f(j)}^i\in Dom N_2$. $N$ is atomic if $RangeN\subseteq \At\A$.
We write $N$ for any of $N, N_1, N_2$ relying on context, we write $nodes(N)$ for $N_1$ and $edges(N)$ for $dom(N_2)$.
$N$ is said to be a network if
\begin{enumerate}[(a)]
\item for all $\bar{x}\in edges(N)$, we have $N(\bar{x})\leq \d_{ij}$ iff $x_i=x_j$.
\item if $\bar{x}\equiv_i \bar{y}$, then $N(\bar{x})\cdot \c_i N (\bar{y})\neq 0.$
\end{enumerate}
\item Let $\A\in TEA_{n}$. A relativized $\A$ pre-network is a pair $N=(N_1, N_2)$
where $N_1$ is a finite set of nodes $N_2:N_1^n\to \A$ is a partial map, such that if $f\in domN_2$,
and $\tau$ is a finite transformation then $\tau| f\in Dom N_2$. Again $N$ is atomic if $RangeN\subseteq \At\A$.
Also we write $N$ for any of $N, N_1, N_2$ relying on context, we write $nodes(N)$ for $N_1$ and $edges(N)$ for $dom(N_2)$.
$N$ is said to be a network if
\begin{enumerate}[(a)]
\item for all $\bar{x}\in edges(N)$, we have $N(\bar{x})\leq \d_{ij}$ iff $x_i=x_j$,
\item if $\bar{x},\bar{y}\in edges(N)$ and $\bar{x}\equiv_i \bar{y}$, then $N(\bar{x})\cdot \c_i N (\bar{y})\neq 0$,
\item $N([i,j]|\bar{x})=\s_{[i,j]}N(\bar{x})$, for all $\bar{x}\in edges(N)$ and all $i,j<n$.
\end{enumerate}
\end{itemize}
\end{defn}
\begin{defn}Let $\A\in PTA_n\cup TEA_n$. We define a game denoted by
$G_{\omega}(\A)$ with $\omega$ rounds, in which the players $\forall$ (male) and $\exists$ (female)
build an infinite chain of relativized $\A$ pre-networks
$$\emptyset=N_0\subseteq N_1\subseteq \ldots.$$
In round $t$, $t<\omega$, assume that $N_t$ is the current prenetwork, the players move as follows:

\begin{enumerate}[(a)]

\item $\forall$ chooses a non-zero element $a\in \A$, $\exists$
must respond with a relativized prenetwork $N_{t+1}\supseteq N_t$ containing an edge $e$ with
$N_{t+1}(e)\leq a$,
\item $\forall$ chooses an edge $\bar{x}$ of $N_t$ and an element $a\in \A$. $\exists$ must respond with a pre-network
$N_{t+1}\supseteq N_t$ such that either $N_{t+1}(\bar{x})\leq a$ or $N_{t+1}(\bar{x})\leq -a$,
\item or $\forall$ may choose an edge $\bar{x}$ of $N_t$ an index $i<n$ and $b\in \A$ with $N_t(\bar{x})\leq {\sf c}_ib$.
$\exists$ must respond with a prenetwork $N_{t+1}\supseteq N_t$ such that for some $z\in N_{t+1},$
$N_{t+1}(\bar{x}^i_z)=b$.
\end{enumerate}
$\exists$ wins if each relativized pre-network $N_0,N_1,\ldots$
played during the game is actually a relativized network.
Otherwise, $\forall$ wins. There are no draws.
\end{defn}
Here we follow closely Hirsch-Hodkinson's techniques adapted to the present situation.

\begin{lem}\label{lem}Let $\A\in PTA_{n}$ be atomic. For all $i,j\in n$, $i\neq j$, define $\t_j^ix=\d_{ij}\cdot \c_i x$ and $\t_i^ix=x$.
Then
\begin{enumerate}[(i)]
\item $(\t_j^i)^{\A}: \At\A\to \At\A$
\item Let $\Omega=\{\t_i^j: i,j\in n\}^*$, where for any set $H$, $H^*$ denotes the free monoid generated by $H$. Let
$$\sigma=\t_{j_1}^{i_1}\ldots \t_{j_n}^{i_n}$$
be a word.
Then define for $a\in A$:
$$\sigma^{\A}(a)=(\t_{j_1}^{i_1})^{\A}((\t_{j_2}^{i_2})^{\A}\ldots (\t_{j_n}^{i_n})^{\A}(a)\ldots ),$$
and $$\hat{\sigma}=[i_1|j_1]| [i_2|j_2]\ldots|[i_n|j_n].$$
Then
$$\A\models \sigma(x)=\tau(x)\text { if } \hat{\sigma}=\hat{\tau}, \sigma,\tau\in \Omega.$$
That is for all $\sigma,\tau\in \Omega$, if $\hat{\sigma}=\hat{\tau}$, then for all $a\in A$, we have
$\sigma^{\A}(a)=\tau^{\A}(a)$.
\end{enumerate}
\end{lem}
\begin{proof} cf. \cite{an} proof of Lemma 1 therein. The $\textbf{MGR}$, merry go round identities
are used here. We note that Andreka's proof of this lemma is long, but using fairly obvious results on semigroups a much shorter proof can be given.
\end{proof}
\begin{lem}\label{soatom}
Let $\A\in PTA_n$ be atomic. For all $i,j\in n$, $i\neq j$, let $\t_j^ix$ be as above.
The following hold for all $i,j,k,l\in n$:
\begin{enumerate}[(i)]
\item $(\t^{i}_{j})^{\A}x\leq\d_{ij}$ for all $x\in\A$.
\item $(x\leq\d_{ij}\Rightarrow(\t^{k}_{i})^{\A}x\leq \d_{ij}\cdot\d_{ik}\cdot\d_{jk})$ for all $x\in\A$.
\item $(a\leq\c_i b\Leftrightarrow\c_i a=\c_i b)$ for all $a,b\in\At\A$.
\item $\c_i (\t^i_j)^{\A} x=\c_i x$ for all $x\in\A$.
\end{enumerate}
\end{lem}
\begin{proof}
\begin{description}
\item[$(i)$] Follows directly from the definition of $(\t^{i}_{j})^{\A}$.
\item[$(ii)$]First we need to check the following for all $i,j,k<n$:
\begin{enumerate}[-]
\item $\c_k\d_{ij}=\d_{ij}$ if $k\notin\{i,j\}$. For, see \cite[Theorem 1.3.3]{hmt1}, the proof doesn't involve $(C_4)$.
\item $\d_{ij}=\d_{ji}$. For, see \cite[Theorem 1.3.1]{hmt1}, the proof works, indeed it doesn't depend on $(C_4)$.
\item $\d_{ij}\cdot\d_{jk}=\d_{ij}\cdot\d_{ik}$. A proof for such can be founded in \cite[Theorem 1.3.7]{hmt1}.
\end{enumerate} Now we can proof ($ii$):
\begin{eqnarray}
\nonumber\t^{k}_{i}x&=&\d_{ik}\cdot\c_k x\\
\nonumber&\leq&\d_{ik}\cdot\c_k\d_{ij}\\
&=&\d_{ik}\cdot\d_{ij}\\
\nonumber&=&\d_{ki}\cdot\d_{ij}\\
&=&\d_{ki}\cdot\d_{kj}
\end{eqnarray}
From $(1)$, $(2)$ the desired follows.
\item[$(iii)-(iv)$] See \cite[p. 675-676]{an}
\end{description}
\end{proof}

\begin{defn}[Partial transposition network]Let $\A$ be an atomic $PTA_n$ and fix an atom $a\in \At\A$. Let $\bar{x}$ be any $n$-tuple (of nodes) such that $x_i=x_j$ if and only if $a\leq \d_{ij}$ for all $i,j<n$. Let $NSQ_{\bar{x}}=\{\bar{y}\in{^{n}\{x_0, x_1, \cdots , x_{n-1}\}}:|Range(\bar{y})|<n\}$ ($NS$ stands for non-surjective sequences). We define the partial transposition network ${PT_{\bar{x}}^{(a)}:NSQ_{\bar{x}}\rightarrow \At\A}$ As follows: If $\bar{y}\in NSQ_{\bar{x}}$, then $\bar{y}=\tr|\bar{x}$, for some  $\si$, $\sj$ $<n$. Let $PT_{\bar{x}}^{(a)}(\bar{y})=\tt a$. This is well defined by Lemma \ref{lem}
\end{defn}
\begin{defn}[Transposition network]
Let $\A$ be an atomic $TEA_n$ and fix an atom $a\in \At\A$. Let $\bar{x}$ be any $n$-tuple of nodes such that $x_i=x_j$ if and only if $a\leq \d_{ij}$ for all $i,j<n$. Let $Q_{\bar{x}}={^{n}\{x_0, x_1, \cdots , x_{n-1}\}}$. Consider the following equivalence relation $\sim$ on $Q_{\bar{x}}$:
\begin{equation*}
\bar{y}\sim\bar{z} \text{ if and only if } \bar{z}=\tau|\bar{y} \text{ for some finite permutation }\tau,
\end{equation*}
$\bar{y},\bar{z}\in Q_{\bar{x}}$.\\
Let us choose and fix representative tuples for the equivalence classes concerning $\sim$ such that each representative tuple is of the form $\tr|\bar{x}$ for some $k\geq 0$, $\si, \sj< n$. Such representative tuples exist. Indeed, for every $\bar{y}\in Q_{\bar{x}}$, $\exists\tau$ finite permutation, $\exists k\geq 0$, $\si$, $\sj<n$ such that $$\bar{y}=\tau|\tr|\bar{x}.$$ Let $\bar{z}=\tau^{-1}|\bar{y}$, then $\bar{z}\sim\bar{y}$ and $\bar{z}=\tr|\bar{x}$. Let $\sf{Rt}$ denote this fixed set of representative tuples. We define the transposition network ${T_{\bar{x}}^{(a)}:Q_{\bar{x}}\rightarrow \At\A}$ as follows:
\begin{itemize}
\item If $\bar{y}\in\sf{Rt}$, then $\bar{y}=\tr|\bar{x}$, for some  $\si$, $\sj$ $<n$. Let $T_{\bar{x}}^{(a)}(\bar{y})=\ttr a$. This is well defined by Lemma \ref{lem}.
\item If $\bar{z}=\sigma|\bar{y}$ for some finite permutation $\sigma$ and some $\bar{y}\in\sf{Rt}$, then let $T_{\bar{x}}^{(a)}(\bar{z})= \s_{\sigma}T_{\bar{x}}^{(a)}(\bar{y})$.
\end{itemize}
\end{defn}
\begin{lem}
The above definition is unique.
\end{lem}
\begin{proof}
The first part is well defined by Lemma \ref{lem}. Now we need to prove that if $\sigma|\bar{y}=\tau|\bar{y}$ for some finite permutations $\sigma, \tau$ and some $\bar{y}\in\sf{Rt}$, then $\s_{\sigma}T^{(a)}_{\bar{x}}(\bar{y})=\s_{\tau}T^{(a)}_{\bar{x}}(\bar{y})$. First, we need the following:
\begin{athm}{Claim}
If $\bar{y}=\tau|\bar{y}$ for some finite permutation $\tau$ and some $\bar{y}\in\sf{Rt}$, then $$T^{(a)}_{\bar{x}}(\bar{y})=\s_{\tau}T^{(a)}_{\bar{x}}(\bar{y}).$$
\end{athm}
\begin{proof}
It suffices to show that if $\bar{y}=[i,j]|\bar{y}$ for some $i, j<n$ and some $\bar{y}\in\sf{Rt}$, then $T^{(a)}_{\bar{x}}(\bar{y})=\s_{[i,j]}T^{(a)}_{\bar{x}}(\bar{y})$. For, suppose that $\bar{y}=[i,j]|\bar{y}$ for some $i, j<n$ and some $\bar{y}\in\sf{Rt}$, then $y_i=y_j$ and then $T^{(a)}_{\bar{x}}(\bar{y})\leq\d_{ij}$ by Lemma \ref{soatom}. Hence by Lemma \ref{lesa}, $T^{(a)}_{\bar{x}}(\bar{y})=\s_{[i,j]}T^{(a)}_{\bar{x}}(\bar{y})$.
\end{proof}
Returning to our prove, assume that $\sigma|\bar{y}=\tau|\bar{y}$ for some finite permutations $\sigma, \tau$ and some $\bar{y}\in\sf{Rt}$. Then $(\tau^{-1}|\sigma)|\bar{y}=\bar{y}$ and so $\s_{(\tau^{-1}|\sigma)}T^{(a)}_{\bar{x}}(\bar{y})=T^{(a)}_{\bar{x}}(\bar{y})$. Therefore, $\s_{\tau^{-1}}\s_{\sigma}T^{(a)}_{\bar{x}}(\bar{y})=T^{(a)}_{\bar{x}}(\bar{y})$. Hence, $\s_{\sigma}T^{(a)}_{\bar{x}}(\bar{y})=\s_{\tau}T^{(a)}_{\bar{x}}(\bar{y})$, and we are done.
\end{proof}
\begin{lem}\label{mynetw}\
\begin{enumerate}
\item Let $\A$ be an atomic $PTA_n$ and $a\in \At\A$. Let $\bar{x}$ be any $n$-tuple of nodes such that $x_i=x_j$ if and only if $a\leq \d_{ij}$ for all $i,j<n$. Then $PT^{(a)}_{\bar{x}}$ is an atomic $\A$ network.
\item Let $\A$ be an atomic $TEA_n$ and $a\in \At\A$. Let $\bar{x}$ be any $n$-tuple of nodes such that $x_i=x_j$ if and only if $a\leq \d_{ij}$ for all $i,j<n$. Then $T^{(a)}_{\bar{x}}$ is an atomic $\A$ network.
\end{enumerate}
\end{lem}
\begin{proof}
Straightforward from the above
.
\end{proof}
\begin{lem} Let $\A\in PTA_n\cup TEA_n$. Then $\exists$ can win any play of $G_{\omega}(\A)$.
\end{lem}
\begin{proof} The proof is similar to that of Lemma 7.8 in \cite{hh}. Let $\A^+$ be the canonical extension of $\A$.
First note that $\A^+\in TEA_n$ and $\A^+$ is atomic. Of course any $\A$ pre-network is an $\A^+$ pre-network.
In each round $t$ of the game $G_{\omega}(\A)$, where $N_t$ is as above, $\exists$ constructs
an atomic $\A^+$ network $M_t$ satisfying
\begin{description}

\item
$M_t\supseteq N_t, nodes(M_t)=nodes(N_t), edges(M_t)=edges(N_t).$
\end{description}
Then if $\bar{x}\equiv_i \bar{y}$, we have
$$N_t(\bar{x})\cdot {\sf c}_iN_t(\bar{y})\geq M_t(\bar{x})\cdot {\sf c}_iM_t(\bar{y})\neq 0.$$
$\exists$ starts by $M_0=N_0=\emptyset$.
Suppose that we are in round $t$ and assume inductively that $\exists$ has managed to construct $M_t\supseteq N_t$ as
indicated above. We consider the possible moves of $\forall$.
\begin{enumerate}[(1)]

\item Suppose that $\forall$ picks a non zero element $a\in \A$. $\exists$ chooses an atom $a^-\in \A^{+}$
with $a^-\leq a$. She chooses new nodes $x_0,\ldots x_{n-1}$ with
$x_i=x_j$ iff $a^-\leq \d_{ij}$.
\begin{description}
\item[If $\A\in PTA_n$.]
She creates two new relativized networks $N_{t+1}$, $M_{t+1}$ with nodes those of $N_t$ plus $x_0,\ldots x_{n-1}$ and hyperedges those of $N_t$ together with $NSQ_{\bar{x}}$. The new hyper labels in
$M_{t+1}$ are defined as follows:
$$M_{t+1}=M_{t}\cup PT^{(a^-)}_{\bar{x}}.$$By Lemma \ref{mynetw} it follows that $M_{t+1}$ is an atomic $\A^+$ network.
Labels in $N_{t+1}$ are given by
\begin{itemize}
\item
$N_{t+1}(\bar{x})=a\cdot \prod_{i,j: x_i=x_j} \d_{ij}$.
\item $N_{t+1}(\bar{y})=\prod_{i,j:y_i=y_j}\d_{ij}$ for any other hyperedge $\bar{y}$.
\end{itemize}
\item[If $\A\in TEA_n$.]
She creates two new relativized networks $N_{t+1}$, $M_{t+1}$ with nodes those of $N_t$ plus $x_0,\ldots x_{n-1}$ and hyperedges those of $N_t$ together with $Q_{\bar{x}}$. The new hyper labels in
$M_{t+1}$ are defined as follows:
$$M_{t+1}=M_{t}\cup T^{(a^-)}_{\bar{x}}.$$By Lemma \ref{mynetw} it follows that $M_{t+1}$ is an atomic $\A^+$ network.
Labels in $N_{t+1}$ are given by
\begin{itemize}

\item
$N_{t+1}(\tau|\bar{x})=\s_{\tau}a\cdot \prod_{i,j: x_i=x_j} {\sf d}_{\tau i\tau j}$ for any finite permutation $\tau$.
\item $N_{t+1}(\bar{y})=\prod_{i,j:y_i=y_j}{\sf d}_{ij}$ for any other hyperedge $\bar{y}$.
\end{itemize}
\end{description}
$\exists$ responds to $\forall$'s move in round $t$ with $N_{t+1}$. One can check that $N_{t}\subseteq N_{t+1}\subseteq M_{t+1}$, as required.

\item If $\forall $ picks an edge $\bar{x}$ of $N_t$ and an element $a\in \A$, $\exists$ lets $M_{t+1}=M_t$ and lets
$N_{t+1}$ be the same as $N_t$ except that\begin{description}
\item[If $\A\in PTA_n$.] $N_{t+1}(\bar{x})=N_t(\bar{x})\cdot a$ if $M_t(\bar{x})\leq a$
and
$N_{t+1}(\bar{x})=N_t(\bar{x}).-a$ otherwise. Because $M_t(\bar{x})$ is an atom in $\A^+$, it follows that if $M_t(\bar{x})\nleq a$, then $M_t(\bar{x})\leq -a$, so this is satisfactory.
\item[If $\A\in TEA_n$.]  for every finite permutation $\tau$, $N_{t+1}(\tau|\bar{x})=N_t(\tau|\bar{x})\cdot \s_{\tau}a$ if $M_t(\tau|\bar{x})\leq \s_{\tau}a$
and
$N_{t+1}(\tau|\bar{x})=N_t(\tau|\bar{x}).-\s_{\tau}a$ otherwise.
Because $M_t(\tau|\bar{x})$ is an atom in $\A^+$ for every finite permutation $\tau$, it follows that if $M_t(\tau|\bar{x})\nleq \s_{\tau}a$, then $M_t(\tau|\bar{x})\leq -\s_{\tau}a$, so this is satisfactory.
\end{description}
\item Alternatively $\forall$ picks $\bar{x}\in N_t$ $i<n$ and $b\in \A$ such that $N_t(\bar{x})\leq {\sf c}_ib$.
Let $M_t(\bar{x})=a^-$. If there is $z\in M_t$ with $M_t(\bar{x}_z^i)\leq b$ then we are done.
In more detail, $\exists$ lets $M_{t+1}=M_t$ and define $N_{t+1}$ accordingly.
Else, there is no such $z$. We have $\c_i a^-\cdot b\not=0$ (inside $\A^+$), indeed
\begin{eqnarray*}
\c_i(\c_i a^-\cdot b)&=&\c_i a^-\cdot \c_i b\text{  }\text{  }\text{  }\text{  }\text{  }\text{  }\text{  }\text{  }\text{  }\text{  }\text{  }\\
&=&\c_i(a^-\cdot \c_i b)\text{  }\text{  }\text{  }\text{  }\text{  }\text{  }\text{  }\text{  }\text{  }\\
&=&\c_i a^-\not=0.\text{  }\text{  }\text{  }\text{  }\text{  }\text{  }(\text{since }a^-\leq\c_i b)
\end{eqnarray*}
Choose an atom $b^-\in\A^+$ with $b^-\leq\c_i a^-\cdot b$. Then we have $b^-\leq b$ and $b^-\leq\c_i a^-$. But by Lemma \ref{soatom} ($iv$) we also have $a^-\leq\c_i b^-$.
\begin{description}
\item[If $\A\in PTA_n$.]Let $G$ be the $\A^+$ network with nodes $\{x_0,\ldots x_{n-1}, z\}$, $z$ a new node,
and the hyperedges are the sequences in $NSQ_{\bar{x}}\cup NSQ_{\bar{t}}$ and $G=PT^{(a^-)}_{\bar{x}}\cup PT^{(b^-)}_{\bar{t}}$
where $\bar{t}=\bar{x}^i_z$.
\item[If $\A\in TEA_n$.]Let $G$ be the $\A^+$ network with nodes $\{x_0,\ldots x_{n-1}, z\}$, $z$ a new node,
and the hyperedges are the sequences in $Q_{\bar{x}}\cup Q_{\bar{t}}$ and $G=T^{(a^-)}_{\bar{x}}\cup T^{(b^-)}_{\bar{t}}$
where $\bar{t}=\bar{x}^i_z$.
\end{description} Again this is well defined. Then, it easy to check that,  $M_t(i_1,\ldots i_n)=G(i_1,\ldots i_n)$ for all $i_1,\ldots i_n\in Range \bar{x}$.
That is the subnetworks of $M_t$ and $G$ with nodes
$Range\bar{x}$ are isomorphic.
Then we can amalgamate $M_t$ and $G$ and define $M_{t+1}$ as the outcome.
The amalgamation here is possible, since the networks are only relativized,
we don't have all hyperedges. By Lemma \ref{mynetw}, one can check that $M_{t+1}$ as so defined is an atomic $\A^+$ network. Now define $N_{t+1}$ accordingly. That is $N_{t+1}$ has the same nodes and edges of $M_{t+1}$, with labelling as for $N_t$ except that
\begin{description}
\item[If $\A\in PTA_n$.]
$N_{t+1}(\bar{t})=b$. The rest of the other labels are  defined to be $$N_{t+1}(\bar{y})=\prod_{i,j:y_i=y_{j}}\d_{ij}.$$
\item[If $\A\in TEA_n$.]
$N_{t+1}(\tau|\bar{t})=\s_{\tau}b$ for any finite permautation $\tau$. The rest of the other labels are  defined to be $$N_{t+1}(\bar{y})=\prod_{i,j:y_i=y_{j}}{\sf d}_{ij}.$$
\end{description}
By Lemma \ref{soatom}, it is clear that $N_{t}\subseteq N_{t+1}\subseteq M_{t+1}$ and then $N_{t+1}$ is $\A$ network. Hence $N_{t+1}$ is appropriate to be played by $\exists$ in response to $\forall$'s move.
\end{enumerate}
\end{proof}
\section{Main results}
\begin{thm}\label{11}\
 Let $2\leq n<\omega$. If $\A\in PTA_{n}$, then $\A\in ID_{n}$.
 \end{thm}\begin{proof}
 Let $\A\in PTA_n$. We want to build an isomorphism from $\A$ to some $\B\in D_{n}$. Consider a play
$N_0\subseteq N_1\subseteq \ldots $ of $G_{\omega}(\A)$ in which $\exists$ plays as in the previous lemma
and $\forall$ plays every possible legal move.
The outcome of the play is essentially a relativized representation of $\A$ defined as follows.
Let $N=\bigcup_{t<\omega}nodes(N_t)$, and $edges(N)=\bigcup_{t<\omega}edges(N_t)\subseteq {}^nN$. By the definition of the networks, $\wp(edges(N))\in D_{n}$. We make $N$ into a representation by
defining $h:\A\rightarrow\wp(edges(N))$ as follows
$$h(a)=\{\bar{x}\in edges(N): \exists t<\omega(\bar{x}\in N_t \& N_t(\bar{x})\leq a)\}.$$
$\forall$-moves of the second kind guarantee that for any $n$-tuple $\bar{x}$ and any $a\in\A$, for sufficiently large $t$ we have either $N_{t}(\bar{x})\leq a$ or $N_{t}(\bar{x})\leq -a$. This ensures that $h$ preserves the boolean operations. $\forall$-moves of the third kind ensure that the cylindrifications are respected by $h$. Preserving diagonals follows from the definition of networks. The first kind of $\forall$-moves tell us that $h$ is one-one. But the construction of the game under consideration ensures that $h$ is onto, too. In fact $\bar{h}$ is a representation from $\A$ onto $\B\in D_n$. This follows from the definition of networks.

 \end{proof}\begin{thm}\label{1} Let $2\leq n<\omega$. If $\A\in TEA_{n}$, then $\A\in IDpe_{n}$.
\end{thm}
\begin{proof}
Let $\A\in TEA_n$. As above consider a play
$N_0\subseteq N_1\subseteq \ldots $ of $G_{\omega}(\A)$ in which $\exists$ plays as in the previous lemma
and $\forall$ plays every possible legal move.
The outcome of the play is essentially a relativized representation of $\A$ defined as follows.
Let $N=\bigcup_{t<\omega}nodes(N_t)$, and $edges(N)=\bigcup_{t<\omega}edges(N_t)\subseteq {}^nN$. Again by the definition of the networks, it is easy to see that $\wp(edges(N))\in Dpe_{n}$. We make $N$ into a representation by
defining $h:\A\rightarrow\wp(edges(N))$ as follows
$$h(a)=\{\bar{x}\in edges(N): \exists t<\omega(\bar{x}\in N_t \& N_t(\bar{x})\leq a)\}.$$
As the previous theorem, $h$ preserves the boolean operations, the cylindrifictions and the diagonals. Now we check transpositions. Let $\bar{y}\in h(\s_{ij}a)$. Then there exists $t<\omega$ such that $N_{t}(\bar{y})\leq\s_{ij} a$. Hence $\s_{ij} N_{t}(\bar{y})=N_{t}([i,j]|\bar{y})\leq a$. The other inclusion is similar. The preservation of the substitutions follows directly from the preservation of the culindrifications and the diagonals.
\end{proof}

\begin{thm}\label{12} Let $2\leq n<\omega$. If $\A\in TA_{n}$, then $\A\in IDp_{n}$.
\end{thm}
\begin{proof}
It is enough to prove that every $TA_\alpha$ is embeddable in a reduct of some $TEA_\alpha$. For, use the same method in the proof of Proposition 9. in \cite{st}, it uses only axioms $(F_3)$, $(F_4)$ and $(F_6)$. This method depend on the fact that $\A$ is definable by positive equations only. So $\A$ is canonical, and we can define diagonals in this canonical extension, $\B$ say, by $\d_{ij}=\bigcap\{y\in B:\s^i_j y=1\}$, for every $i,j\in\alpha$. Then it can be shown that $B$ with those constants is in $TEA_{\alpha}$.
\end{proof}

\begin{thm}\label{13} Let $2\leq n<\omega$. If $\A\in SA_{n}$, then $\A\in IDs_{n}$.
\end{thm}
\begin{proof}
We will use the same analogue in the proof of Proposition 9. in \cite{st}. Now, let $2\leq n<\omega$ and $\A\in SA_{n}$. Then $\A$ is definable by positive equations only. Indeed, There is an axiomatization of Boolean algebra involving only meets and joins, so a Boolean homomorphism is specified by respecting meets and joins. By this we get rid of negation. Therefore by (I) in \cite[p. 440]{hmt1}, $\A$ is a subalgebra of a complete and atomic $\B$ such that $\B\models(C_0 - C_7)$. Let $\d_{ij}=\bigcap\{y\in B:\s^i_j y=1\}$, for every $i,j\in\alpha$. This definition is justified because $\B$ is complete. Our aim is to prove that $\B$ with this constants satisfies the axioms $C_0-C_7$ and this finishes the prove. For, it is enough to prove that $\B\models(C_5-C_7)$.
\begin{cl}
For every $i,j\in\alpha$ and every set $K$, $\s^i_j(\bigcap_{k\in K}y_k)=\bigcap_{k\in K}\s_j^i y_k$.
\end{cl}
\begin{proof}
See \cite[Claim 9.1.]{st}
\end{proof}
\begin{cl}
For every $i,j\in\alpha$, $\B\models\s^i_j\d_{ij}=1$.
\end{cl}
\begin{proof}Let $i,j\in\alpha$,
\begin{eqnarray*}
\B\models \text{ }\text{ }\text{ }\text{ }\text{ }\text{ }\s^i_j\d_{ij}&=&\s^i_j\{y\in B: \s^i_j y=1\}\\
&=&\{\s^i_j y : y\in B \text{ and } \s^i_j y=1\}\\
&=&1.
\end{eqnarray*}
\end{proof}
\begin{cl}
For every $i,j\in\alpha$ and every $x\in B$, if $i\neq j$ then $\B\models\s^i_jx=\c_i(x\cdot\d_{ij})$.
\end{cl}
\begin{proof}
Let $y\in B$ be such that $\s^i_j y=1$.
\begin{eqnarray*}
\B\models \text{ }\text{ }\text{ }\text{ }\text{ }\text{ }\text{ }\text{ } \s^i_j[-(x\cdot y)+\s^i_j x]&=& -(\s^i_j x\cdot \s^i_j y)+\s^i_j\s^i_j x\\
&=& -\s^i_j x+\s^i_j x\\
&=& 1,
\end{eqnarray*}
therefore, $\B\models \text{ }\text{ }\text{ }\text{ }x\cdot y\leq \s^i_j x$.
Hence, \begin{eqnarray*}
\B\models \text{ }\text{ }\text{ }\text{ }\text{ } \bigcap\{x\cdot y:y\in B\text{ and }\s^i_j y=1\}&\leq& \s^i_j x\\
\B\models \text{ }\text{ }\text{ }\text{ }\text{ } x\cdot\d_{ij}\leq\s^i_j x,
\end{eqnarray*}
i.e., $\B\models \text{ }\text{ }\text{ }\text{ }\text{ }\c_i(x\cdot\d_{ij})\leq\s^i_j x$.
On the other direction, \begin{eqnarray*}
\B\models \text{ }\text{ }\text{ }\text{ }\text{ } x\cdot\d_{ij}&\leq&\c_i(x\cdot\d_{ij})\\
\B\models \text{ }\text{ }\text{ }\text{ }\text{ } \s^i_j(x\cdot\d_{ij})&\leq&\s^i_j(\c_i(x\cdot\d_{ij}))\\
\B\models \text{ }\text{ }\text{ }\text{ }\text{ } \s^i_j x\cdot\s^i_j\d_{ij}&\leq&\c_i(x\cdot\d_{ij})\\
\B\models \text{ }\text{ }\text{ }\text{ }\text{ } \s^i_j x&\leq&\c_i(x\cdot\d_{ij}).
\end{eqnarray*}
\end{proof}
Now, \begin{eqnarray*}
\B\models\d_{ij}&=&\{y\in A: \s^i_i y=1\}\\
&=&\{y\in A: y=1\}\\
&=& 1.
\end{eqnarray*}

\begin{eqnarray*}
\B\models&&\s^i_j[-(\d_{ij}\cdot\c_i(\d_{ij}\cdot x))+x]\\
&=& -(\s^i_j\d_{ij}\cdot\s^i_j\c_i(\d_{ij}\cdot x))+\s^i_j x\\
&=& -\c_i(\d_{ij}\cdot x)+\s^i_j x\\
&\geq& -\s^i_j x+\s^i_j x\\
&=&1.
\end{eqnarray*}
Therefore, $\B\models \d_{ij}\cdot\c_i(\d_{ij}\cdot x)\leq x$.
 
\begin{eqnarray*}
\B\models&&\s^i_j[-\c_k(\d_{ik}\cdot\d_{kj})+\d_{ij}]\\
&=& \s^i_j[-\s^k_i\d_{kj}+\d_{ij}]\\
&=&-\s^i_j\s^k_i\d_{kj}+\s^i_j\d_{ij}\\
&=& -\s^i_j\s^k_j\d_{kj}+1\\
&=&1.
\end{eqnarray*}
Therefore, $\B\models \c_k(\d_{ik}\cdot\d_{kj})\leq\d_{ij}$.
Hence we proved that $\B\models (C_0-C_7)$ and also $\B\models\s^i_j x=\c_{i}(\d_{ij}\cdot x)$. This finishes the prove.

\end{proof}
We have proved our theorems for finite dimensions. Now we turn to the
infinite dimensional case. We give a general method of lifting
representability for finite dimensional cases to the transfinite; that
can well work in other contexts.
\begin{thm}\label{2} Assume that $\alpha\geq \omega$. Then $PTA_{\alpha}=ID_{\alpha}$, $TEA_{\alpha}=IDpe_{\alpha}$, $TA_{\alpha}=IDp_{\alpha}$ and $SA_{\alpha}=IDs_{\alpha}$.
\end{thm}
\begin{proof} We will consider the case of $PTA_{n}$ and the other cases are similar. First, we know that $PTA_{n}=ID_{n}$ for every finite $n<\omega$. We want to show that $PTA_{\alpha}=ID_{\alpha}$ for any infinite $\alpha$. First note that :
\begin{enumerate}
\item For any $\A\in PTA_{\alpha}$ and $\rho:n\rightarrow\alpha$, $n\in\omega$ and $\rho$ is one to one, define $\mathfrak{Rd}^{\rho}\A$ as in \cite[Definition 2.6.1]{hmt1}. Then $\mathfrak{Rd}^{\rho}\A\in PTA_{n}$.
\item For any $n\geq 2$ and $\rho:n\rightarrow\alpha$ as above, $ID_{n}\subseteq S\mathfrak{Rd}^{\rho}ID_{\alpha}$ as in \cite[Theorem 3.1.121]{hmt2}.
\item $ID_{\alpha}$ is closed under the ultraproducts, cf. \cite[Lemma 3.1.90]{hmt2}.
\end{enumerate}

Now we show that if $\A \in PTA_{\alpha}$, then $\A$ is representable. First, for any
$\rho:n\rightarrow\alpha$, $\mathfrak{Rd}^{\rho}\A\in PTA_{n}$. Hence $\mathfrak{Rd}^{\rho}\A$ is in $ID_n$ and so it is in $S\mathfrak{Rd}^{\rho}ID_{\alpha}$. Let
$J$ be the set of all finite one to one sequences with range in $\alpha$. For $\rho\in J$, let
$M_{\rho}=\{\sigma\in J: \rho\subseteq\sigma\}$. Let $U$ be an ultrafilter of $J$ such that $M_{\rho}\in U$ for
every $\rho\in J$. Then for $\rho\in J$, there is $\B_{\rho}\in ID_{\alpha}$ such that $\mathfrak{Rd}^{\rho}\A\subseteq \mathfrak{Rd}^{\rho}\B_{\rho}$. Let $\mathfrak{C}=\prod \B_{\rho}/U$; it is in $UpID_{\alpha}=D_{\alpha}$. Define $f:\A\rightarrow\prod\B_{\rho}$ by $f(a)_{\rho}=a$, and finally define $g:\A\rightarrow\mathfrak{C}$ by $g(a) = f(a)/U$. Then $g$ is an embedding.

\end{proof}
\begin{remark}Let $\A$ be an algebra in some class of our interest. 
$\A$ is said to be completely representable if there is a representation $f:\A\to P(V)$ such that $\bigcup \{f(x): x \text{ an atom}\} =V$.\footnote{A characterization of the completely representable algebras, cf, \cite[Lemma 2.1]{hod}} Therefore, according to the representations that are built in the proofs of Theorems \ref{11}, \ref{1}, \ref{12}, \ref{13} and \ref{2}, every atomic algebra in $PTA_{\alpha}\cup TEA_{\alpha}\cup TA_{\alpha}\cup SA_{\alpha}$ is completely representable.
\end{remark}
Here we compare our classes with other important classes existing in the literature.
Given a set $U$ and a mapping $p\in{^{\alpha}U}$, then the set $${^{\alpha}U^{(p)}}=\{x\in {^{\alpha}U}: x \text{ and }p \text{ are different only in finitely many places}\}$$ is called the weak space determined by $p$ and $U$.
\begin{defn}
\begin{description}
\item[Class $Gw_{\alpha}$] A set algebra in $Crs_{\alpha}$ is called a generalized weak cylindric relativized set algebra if there are sets $U_k$, $k\in K$, and mappings $p_k\in{^{\alpha}U_k}$ such that $\V=\bigcup_{k\in K}{^{\alpha}U^{(p_k)}_{k}}$, where $\V$ is the unit.
\item[Class $Gwp_{\alpha}$ ($Gwpe_{\alpha}$)] A set algebra in $Prs_{\alpha}$ ($Pers_{\alpha}$) is called a generalized weak polyadic (equality) relativized set algebra if there are sets $U_k$, $k\in K$, and mappings $p_k\in{^{\alpha}U_k}$ such that $\V=\bigcup_{k\in K}{^{\alpha}U^{(p_k)}_{k}}$, where $\V$ is the unit.
\item[Class $Gws_{\alpha}$] A set algebra in $Srs_{\alpha}$ is called a generalized weak substitution relativized set algebra if there are sets $U_k$, $k\in K$, and mappings $p_k\in{^{\alpha}U_k}$ such that $\V=\bigcup_{k\in K}{^{\alpha}U^{(p_k)}_{k}}$, where $\V$ is the unit.
\end{description}
\end{defn}
A known characterization of the class $Gwpe_{\alpha}$ is : If $\V$ is the unit of an $\sA\in Prs_{\alpha}$, then $\sA\in Gwp_{\alpha}$ if and only if $y\in\V$ implies $\tau|y\in\V$ for every finite transformation $\tau$. Using this property one can prove that $Gwpe_{\alpha}=Dpe_{\alpha}$, so we can replace $Dpe_{\alpha}$ by $Gwpe_{\alpha}$ in Theorem \ref{1} and Theorem \ref{2}. But the same is not true for the other types, for example, for finite $n\in\omega$, the class $Gw_{\alpha}$ coincide with the class of locally square cylindric algebras. Andreka gave a finite schema axiomatization for the former class in \cite{gn}, and Andreka's axioms and $PTA_{\alpha}$ are not definitionally equivalent.

\section{All varieties considered have the superamalgamation property}

\begin{defn} Let $K$ be a class of algebras having a boolean reduct. 
$\A_0\in K$ is in the amalgamation base of $K$ if for all $\A_1, \A_2\in K$ and monomorphisms $i_1:\A_0\to \A_1,$ $i_2:\A_0\to \A_2$ 
there exist $\D\in K$
and monomorphisms $m_1:\A_1\to \D$ and $m_2:\A_2\to \D$ such that $m_1\circ i_1=m_2\circ i_2$. 
If in addition, $(\forall x\in A_j)(\forall y\in A_k)
(m_j(x)\leq m_k(y)\implies (\exists z\in A_0)(x\leq i_j(z)\land i_k(z) \leq y))$
where $\{j,k\}=\{1,2\}$, then we say that $\A_0$ lies in the super amalgamation base of $K$. Here $\leq$ is the boolean order.
$K$ has the (super) amalgamation property $((SUP)AP)$, if the (super) amalgamation base of $K$ coincides with $K$.
\end{defn}

We now show using a result of Marx, that all varieties considered have the superamalgmation property $(SUPAP).$
We consider $TA_{\alpha}=ID_{\alpha}$ The rest of the cases are the same.
For a set $\V$, $B(\V)$ denotes the Boolean algebra $(\wp(\V), \cap, \sim).$ 

\begin{defn} 
\begin{enumerate}[1.]

\item A frame of type $TA_{\alpha}$ is a first order structure $$F=(\V, \C_i, \S_{i}^j, \S_{ij})_{i,j\in \alpha},$$ where $\V$ is an arbitrary set and
and $\C_i$, $\S_i^j$ and $\S_{ij}$ are binary relations for all $i, j\in \alpha$.

\item Given a frame $F$, its complex algebra denote by $F^+$ is the algebra $$(B(\V), \c_i, \s_i^j, \s_{ij})_{i,j},$$  where for $X\subseteq \V$, 
$\s_i^j(X)=\{\s\in \V: \exists t\in X, (t, s)\in \s_i^j \}$, and same for $\c_i$ and $\s_{ij}.$

\item Given $K\subseteq TA_{\alpha},$ then $Cm^{-1}K=\{F: F^+\in K\}.$

\item Given a family $(F_i)_{i\in I}$ a zigzag product of these frames is a substructure of $\prod_{i\in I}F_i$ such that the 
projection maps restricted to $\S$ are
 onto. 

\end{enumerate}

\end{defn}

\begin{thm}(Marx) 
Assume that $K$ is a canonical variety and $L=Cm^{-1}K$ is closed under finite zigzag products. Then $K$ has the superamalgamation 
property.
\end{thm}
\begin{proof}
See \cite[Lemma 5.2.6 p. 107.
]{marx}.
\end{proof}
\begin{thm}
The variety $TA_{\alpha}$ has $SUPAP$.
\end{thm}
\begin{proof}  Since $TA_{\alpha}$ is defined by positive equations then it is canonical. 
In this case $L=Cm^{-1}TA_{\alpha}$ consists of frames $(\V, \C_i, \S_i^j, \S_{i,j})$
such that if $s\in \V$, then $[i,j]|s\in \V$ and $[i|j]|s$ is in $\V$. Moreover, $(x,y)\in \C_i$ iff $x$ and $y$ agree off $i$,
 $(x,y)\in \S_i^j$ iff $[i|j]|x= y$ and same for $\S_{ij}$.
The first order correspondants of the positive equations translated to the class of frames will be Horn formulas, hence clausifiable 
and so $L$ is closed under finite zigzag products. 
\end{proof}

\begin{athm}{ACKNOWLEDGEMENT} We are grateful to Mikl\'os Ferenczi, for his fruitful discussion.
\end{athm}

\end{document}